\documentclass[10pt]{amsart}
\usepackage{appendix}
\usepackage{mathtools}
\usepackage{graphicx,enumitem,dsfont}
\usepackage[colorlinks]{hyperref}
\newcommand{\Z}{{\mathbb{Z}} }
\newcommand{\N}{{\mathbb{N}}}
\newcommand{\R}{\mathbb{R}}
\newcommand{\Pol}{\mathbb{P}}

\renewcommand{\i}{\ifmmode\mathit{\mathchar"7010 }\else\char"10 \fi}
\renewcommand{\j}{\ifmmode\mathit{\mathchar"7011 }\else\char"11 \fi}

\newcommand{\seq}[1]{\left\{#1\right\}}
\newcommand{\test}{\varphi}
\newcommand{\Dx}{{\Delta x}}
\newcommand{\Dt}{{\Delta t}}
\newcommand{\norm}[1]{\left\|#1\right\|}
\newcommand{\abs}[1]{\left|#1\right|}

\DeclareMathOperator*{\csch}{csch}
\DeclareMathOperator*{\sech}{sech}

\newtheorem{definition}{Definition}[section]

\newtheorem{theorem}{Theorem}[section]
\newtheorem{lemma}{Lemma}[section]

\newtheorem{remark}{Remark}[section]
\theoremstyle{definition} 

\newtheorem*{maintheorem*}{Main Theorem 1}
\newtheorem*{maintheorem**}{Main Theorem 2}

\allowdisplaybreaks

\title[KdV]{ Convergence of a Higher-Order scheme for Korteweg-de Vries equation}

\author[Dutta]{Rajib Dutta}
\address[Rajib Dutta]
{\newline
  Centre of Mathematics for Applications, Department of Mathematics,
  University of Oslo, P.O.\ Box 1053, Blindern, NO--0316 Oslo, Norway}
\email[]{rajibd@math.uio.no}

\author[Koley]{Ujjwal Koley} \address[Ujjwal Koley] {\newline  
   Tata Institute of Fundamental Research Centre,  
   Centre For Applicable Mathematics,   
   Post Bag No. 6503, GKVK Post Office,  Sharada
   Nagar, Chikkabommasandra, 
   Bangalore 560065, India.} 
\email[]{ujjwal@math.tifrbng.res.in}

\author[Risebro]{Nils Henrik Risebro} \address[Nils Henrik Risebro]{\newline
  Centre of Mathematics for Applications, Department of Mathematics,
  University of Oslo, P.O.\ Box 1053, Blindern, NO--0316 Oslo, Norway}
\email[]{nilshr@math.uio.no}
\urladdr{\href{http://www.mn.uio.no/math/english/people/aca/nilshr/}%
  {http://www.mn.uio.no/math/english/people/aca/nilshr/}}

\date{\today}

\begin{document}

\maketitle
\begin{abstract}
  We are study the convergence of higher order schemes for the Cauchy
  problem associated to the KdV equation. More precisely, we design a
  Galerkin type implicit scheme which has higher order accuracy in
  space and first order accuracy in time.  The convergence is
  established for initial data in $L^2$, and we show that the scheme
  converges strongly in $L^2(0,T;L^2_{\text{loc}}(\R))$ to a weak
  solution. Finally, the convergence is illustrated by several
  examples.
\end{abstract}

\section{Introduction}
In this paper, we consider a higher order finite element Galerkin type
scheme for computing approximate solutions of the Cauchy problem for
Korteweg-de Vries (KdV) equation
\begin{equation}
  \begin{cases}
    \label{eq:main}
    u_t + (\frac {u^2} 2)_x + u_{xxx} =0, & \quad x \in \R \times (0,T) \\
    u(x,0) =u_0(x), & \quad x \in \R,
  \end{cases}
\end{equation}
where $T>0$ is fixed, $u: \R \times [0,T) \rightarrow \R$ is the
unknown, and $u_0$ the initial data.

It is well known that the KdV equation models the propagation of waves
of small amplitude in dispersive systems (e.g.,  magneto-acoustic waves in
plasmas, shallow water waves, lattice waves and so on). Also, the KdV
equation has localized solutions, i.e., solutions whose value approach
a constant for $\abs{x}$ large, called solitons. These have the
property that their speed increases with their amplitude, and that
such solitary waves interact in a particle like manner. 

The first mathematical proof of existence and uniqueness of solutions
of the KdV equation was accomplished by Sj\"{o}berg
\cite{Sjoberg:1970} in 1970, using a semi-discrete finite difference
approximation, where one discretizes the spatial variable, thereby
reducing the equation to a system of ordinary differential equations.

Well posedness for the KdV equation has been studied extensively in
the last three decades, see \cite{Tao:2006, LinaresPonce:2009} and the
references therein. We will not discuss the vast literature regarding
the mathematical properties of the KdV equation here, but mention that
local well posedness local is proved in the Sobolev spaces $H^s$ for
$s>-3/4$ in \cite{kenig:1996}.

On the other hand, numerical computations for the KdV
equation has also been of great interest, since the landmark work by
Zabusky and Kruskal \cite{ZabuskyKruskal:1965}, where they discovered
the permanence of solitons for KdV equation using numerical
techniques. In fact, the numerical computation of solutions of the KdV
equation is rather capricious.  Two competing effects are involved,
namely the nonlinear convective term $uu_x$, which in the context of
the Burgers equation $u_t+uu_x= 0$ yields infinite gradients in finite
time even for smooth data, and the linear dispersive term $u_{xxx}$,
which in the Airy equation $u_t+u_{xxx}= 0$ produces hard-to-compute
dispersive waves, and these two effects combined makes it difficult to
obtain accurate and fast numerical methods.

There are number of numerical schemes available to analyze the behaviour
of solutions to the KdV equation numerically. We will 
discuss the full literature here, but only refer to those results which are
relevant to this paper. 

Spectral methods have been studied
extensively, see \cite{NouriSloan:1989,KassamTrefethen:2005} and
references therein.  Multi-symplectic schemes have been studied in
\cite{AscherMcLachlan:2005} (see also references therein).  Standard
Galerkin type approximations, using smooth splines on a uniform mesh,
to periodic solutions of KdV equation are analyzed in
\cite{ArnoldWinther:1982, Baker:1983, Winther:1980} .  All these work
aimed at deriving optimal rate of convergence estimate for Galerkin
approximations.  The discontinuous Galerkin method has been used to
approximate the solution of \eqref{eq:main} and rate of convergence
analysis has been presented for both periodic and full line case in
\cite{ShuYan:2002}. 

All the above mentioned references use the well posedness theory for
the KdV equation to prove convergence, and convergence
rates. Therefore, by themselves, they do not yield the existence of a
solution by furnishing constructive existence proofs.


There are however a few results regarding proof of convergence of
numerical methods for the KdV equation, which also give a direct and
constructive existence theorem. Indeed, the first proof of existence
and uniqueness of solutions to the KdV equation for initial data in
$H^3(\R/\Z)$ is based on a semi-discrete difference approximation
\cite{Sjoberg:1970}.  The corresponding fully discrete scheme, which
incidentally coincides with a fully discrete splitting scheme, was
analyzed in \cite{HKR1}, and it was shown that the scheme converges to
the classical solution if the initial data is in $H^3(\R)$, and to the
weak solution if the initial data lies in $L^2(\R)$.  The proof
assumes the CFL condition $\Dt=\mathcal{O}(\Dx^2)$ where $\Dt$ and
$\Dx$ are the temporal and the spatial discretizations
respectively. Laumer proved the direct convergence of a similar
scheme, but under the improved CFL condition
$\Dt=\mathcal{O}(\Dx)$. The results in this paper can be seen as a
generalization of the above in the context of higher order
approximation methods.

Our main tool is an observation due to Kato. In \cite{kato} it was proved that
the solution operator of the KdV equation has a smoothing effect due
to the dispersion. This smoothing permits a proof of existence of
solutions if the initial data are only in $L^2(\R)$. The smoothing
effect inherent in the KdV equation is
not as strong as for parabolic equations, and is of course absent in
the case of hyperbolic conservation laws. Precisely, solutions of
\eqref{eq:main} satisfy
\begin{equation}\label{eq:Katoestimate}
  \int_{-T}^T\int_{-R}^R \abs{u_x}^2 dxdt \le C(T,R).
\end{equation}
An analogue of this estimate is the main ingredient in our proof of
the convergence of our approximate solutions $u_\Dx$. 

The approximation $u_\Dx$ is generated by an implicit Euler
discretization of a Galerkin scheme with approximations in a subspace
of $H^2(\R)$ consisting of piecewise polynomial functions. Inspired by 
the proof of \eqref{eq:Katoestimate} we define the Galerkin
approximations using a weight function $\test$, which is positive and
constant outside an interval $(-Q,Q)$. Using this in our scheme
enables us to prove that
the collection $\seq{u_\Dx}_{\Dx>0}$ lies in the set
\begin{equation*}
  W=\seq{w\in L^2(0,T; H^1([-R,R]))\,\mid \, w_t\in L^2(0,T; H^{-2}([-R,R]))},
\end{equation*}
which is compact in  $L^2(0,T; L^2([-R,R]))$ by the
Aubin--Simon compactness lemma.


The rest of the paper is organized as follows: In
Section~\ref{sec:scheme}, we present the necessary notation and define
the fully-discrete finite element Galerkin type numerical
scheme. Since the fully-discrete scheme is implicit in nature, the
solvability of the scheme cannot be taken for granted and this is
addressed in Section~\ref{subsec:sol}.  In Section~\ref{sec:convergence},
we show the convergence to a weak solution if the initial data is in
$L^2(\R)$ and finally in Section~\ref{sec:numerical}, we exhibit some
numerical experiments showing the convergence.

\section{Numerical Scheme}
\label{sec:scheme}
We start by introducing some notation needed to define the Galerkin
finite element scheme. Throughout this paper we reserve $\Dx$ and
$\Dt$ to denote two small positive numbers that are the spatial
and temporal discretization parameters respectively, of the numerical
scheme.

For $j \in \N_{0} = \N \cup \lbrace 0 \rbrace$, we set $x_j=j\Dx$, and
for $n =0,1,\cdots, N$, where $N\Dt=T$ for some fixed time horizon
$T>0$, we set $t_n = n \Dt$.  Furthermore, we introduce the spatial
grid cells $I_j=[x_{j-1}, x_j]$.

Moreover given $R>0$, we define the cut off function $\varphi$ as
$\varphi(x)=\overline{\varphi} *w(x)$ where $\overline{\varphi}
(x)=\max\seq{1,\min\seq{1+x+R, 1+2R}}$ and $w$ is a symmetric positive
function with integral one and support in $[-1,1]$.  
Let $C_R$ be defined as
\begin{equation}\label{eq:CRdef}
  C_R=\max\seq{\norm{\varphi}_{L^\infty(\R)},
    \norm{\varphi_x}_{L^\infty(\R)},\norm{\varphi_{xx}}_{L^\infty(\R)},
    \norm{\varphi_{xxx}}_{L^\infty(\R)}}. 
\end{equation}
We define the weighted $L^2$ inner product as
\begin{equation*}
  \langle u,v\rangle_{\test} = \left(u,v\test\right)
\end{equation*}
where $(\cdot,\cdot)$ denotes the usual $L^2$ inner product,
and the associated weighted norm by $\norm{u}_{2,\test}^2=\langle
u,u\rangle_\test$.  

\subsection{Variational Formulation}
We assume that $r$ is a fixed integer $\geq 2$ and let $\Pol_r(I)$
denotes the space of polynomials on the interval $I$ of degree $\le
r$.  We seek an approximation $u$ to the solution of \eqref{eq:main}
such that for each $t \in [0,T]$, $u$ belongs to the finite
dimensional space
\begin{equation*}
  S_{\Dx}=\seq{ v\in H^2(\R)\, \mid\, v\in \Pol_r(I_j) \ \ \text{for all} \
    j}.
\end{equation*}
The variational form is derived by multiplying the strong form
\eqref{eq:main} with test functions $\varphi v $, with $ v\in S_{\Dx}$
and $\varphi$ specified above, and integrating over each element
separately. After integrating by parts twice, we obtain
\begin{align*}
  \left(u_t, \varphi v\right) - \left(\frac{u^2}{2}, (\varphi
    v)_x\right) + \left(u_x,(\varphi v)_{xx}\right) = 0, \hspace{.2cm}
  \forall v\in S_{\Dx}
\end{align*}
This is the semi-discrete form of the variational
formulation. However, in order to have a practical numerical method,
we must use a numerical method to integrate in time. We use the implicit
Euler method for this. This scheme reads as follows:
Find $u^n
\in S_{\Dx}$ such that
\begin{equation}
  \label{eq:scheme}
  \left(u^{n+1}, \varphi v\right)-\left(u^{n}, \varphi v\right) -
  \Dt \left(\frac{(u^{n+1})^2}{2}, (\varphi v)_x\right) + \Dt
  \left(u^{n+1}_x,(\varphi v)_{xx}\right) = 0,
\end{equation}
for all $v\in S_\Dx$ and
for $n=0,1,\ldots$, with initial data given by $u^0=Pu_0$,
where $P$ is the $L^2(\R)$ orthogonal projection onto $S_\Dx$.
Observe that, this is an implicit scheme, and in order to calculate
$u^{n+1}$ given $u^n$ one must solve a non-linear equation.

\subsection{Solvability for one time step }
\label{subsec:sol}
To solve \eqref{eq:scheme}, we use a simple fixpoint iteration, and define the
sequence $\seq{w^{\ell}}_{\ell\ge 0}$ by letting $w^{\ell+1}$ be the
solution of the linear equation
\begin{equation} \label{eq:iteration scheme}
  \begin{cases}
    \left(w^{\ell+1}, \varphi v\right) + \Dt \left(w^\ell w^\ell_x, \
      \varphi v\right)
    +  \Dt \left(w^{\ell+1}_x,(\varphi v)_{xx}\right)  = (u,\varphi v),  \\
    w^0 =u^n,
  \end{cases}
\end{equation}
this is to hold for all test functions $v\in S_\Dx$. The following
lemma guarantee the solvability of the implicit scheme
\eqref{eq:scheme}.
\begin{lemma}
  \label{lemma1}
  Choose a constant $L$ such that $0<L<1$ and set
  \begin{equation*}
    K=\frac{7-L}{1-L}>7.
  \end{equation*}
  We consider the iteration \eqref{eq:iteration
    scheme} with $w^0=u^n$, and assume that the following CFL
  condition holds
  \begin{equation}\label{eq:cfl}
    \lambda\le \frac{L}{\sqrt{C_R} 2\sqrt{2} K \norm{u^n}_{2,\varphi}},
  \end{equation}
  where where $C_R$ is defined by \eqref{eq:CRdef} and $\lambda$ is given by
  \begin{equation}
    \label{eq:lambdadef}
    \lambda^2 = \frac{\Dt^2}{\Dx^3}.
  \end{equation}
  Then there exists a function $u^{n+1}$ which solves
  \eqref{eq:scheme}, and $\lim_{\ell\to\infty}w^\ell = u^{n+1}$.
  Furthermore
  \begin{equation}
    \label{eq:unp1bnd}
    \norm{u^{n+1}}_{2,\varphi}\le K \norm{u^n}_{2,\varphi}.
  \end{equation}
\end{lemma} 
\begin{proof}
  From \eqref{eq:iteration scheme} we have
  \begin{equation}
    \label{eq:iterationdiff}
    \left(w^{\ell+1}-w^\ell, \varphi v\right) + \Dt \left( w^\ell w^\ell_x
      -w^{\ell-1}w^{\ell-1}_x , 
      \varphi v \right) + \Dt (w^{\ell+1}_x -w^\ell_x, (\varphi v)_{xx}) =
    0
  \end{equation}
  for any $v \in S_{\Dx}$.  We choose $v=w^{\ell+1}-w^\ell$ in
  \eqref{eq:iterationdiff} to get
  \begin{align*}
    \langle v,v\rangle_\test &+ \Dt \left(v_x,(\test v)_{xx}\right) =
    -\Dt \left(w^\ell w^\ell_x -
      w^{\ell-1}w^{\ell-1}_x,\varphi(w^{\ell+1} -w^\ell)\right)
    \\
    &\qquad \le \frac{1}2 \langle v,v\rangle_\test + \frac{\Dt^2}2
    \left(w^\ell w^\ell_x - w^{\ell-1}w^{\ell-1}_x,\test\left(w^\ell
        w^\ell_x - w^{\ell-1}w^{\ell-1}_x\right)\right),
  \end{align*}
  by Young's inequality. Therefore
  \begin{align*}
  &  \frac12\norm{v}^2_{2,\test} + \Dt
    \left(v_x,\left(v\test\right)_{xx}\right) \le \frac{\Dt^2}2 \int_\R
    \left( w^\ell w^\ell_x
      -w^{\ell-1}w^{\ell-1}_x \right)^2 \varphi\, dx
    \\
    &\quad =\frac{\Dt^2}{2}\int_\R \left(\left(w^\ell -
        w^{\ell-1}\right)w^\ell_x -
      w^{\ell-1}\left(w^\ell - w^{\ell-1}\right)_x\right)^2\varphi\,dx
    \\
    &\le \Dt^2\int_\R \left(w^\ell-w^{\ell-1}\right)^2
    \left(w^\ell_x\right)^2\varphi\,dx + \Dt\int_\R
    \left(w^{\ell-1}\right)^2\left(w^\ell_x-w^{\ell-1}_x\right)^2\varphi\,dx
    \\
    &\le \Dt^2 \norm{w^{\ell}_x}_{L^\infty(\R)}^2 \int_\R
    \left(w^\ell-w^{\ell-1}\right)^2\varphi\,dx + \Dt \norm{w^\ell_x -
      w^{\ell-1}_x}_{L^\infty(\R)}^2 \int_\R
    \left(w^{\ell-1}\right)^2\varphi\,dx
    \\
    &\le \frac{C\Dt^2}{\Dx^{3}} \norm{w^\ell}_{L^2(\R)}^2\int_\R
    \left(w^\ell-w^{\ell-1}\right)^2\varphi\,dx + \frac{C\Dt}{\Dx^{3}}
    \norm{w^\ell-w^{\ell-1}}_{L^2(\R)}^2 \int_\R \left( w^{\ell-1}\right)^2
    \varphi\,dx 
    \\
    &\le C_{R} \lambda^2 \max\seq{\norm{w^\ell\sqrt{\varphi}}_{L^2(\R)}^2,
      \norm{w^{\ell-1}\sqrt{\varphi}}_{L^2(\R)}^2} \int_\R
    \left(w^\ell-w^{\ell-1}\right)^2\varphi\,dx ,
  \end{align*}
    We have used the following ``inverse inequalities''; for any
  function $z\in S_{\Dx}$
  \begin{equation}
    \label{eq:inverse_inequal}
    \norm{z_x}_{L^\infty(\R)} \le \frac C{\Dx^{1/2}} \norm{z_x}_{L^2(\R)} \le
    \frac C{\Dx^{3/2}} \norm{z}_{L^2(\R)},
  \end{equation}
  where the constant $C$ is independent of $z$ and $\Dx$.
  To take care the second term on the left, we
  make an use of the following identity
  \begin{align}
    \label{eq:identity}
    \int_{\R}w_x\, (\varphi w)_{xx} \ dx = \frac32
    \int_{\R}w_x^2\varphi_{x}\ dx\, -\, \frac12
    \int_{\R}w^2\varphi_{xxx}\, dx,
  \end{align}
  which is established by repeated use of integration by parts.
  Thus
  \begin{multline*}
    \int_\R \left(w^{\ell+1}-w^\ell\right)_x
    ((w^{\ell+1}-w^\ell)\varphi)_{xx}\,dx \\
    =
    \frac{3}2\int_\R \left(w^{\ell+1}_x-w^\ell_x\right)^2 \varphi_x\,dx 
    -\frac{1}2 \int_\R \left(w^{\ell+1}-w^\ell\right)^2\varphi_{xxx}\,dx 
  \end{multline*}
  Since $\varphi_x\ge 0$, the second term after the above inequality is
  non-negative, and we have
  \begin{equation*}
    \Dt
    \left(v_x,\left(v\test\right)_{xx}\right) \ge -C_R \Dt\norm{v}_{L^2(\R)}^2
    \ge -C_R\Dt \norm{v}_{2,\test}^2,
  \end{equation*}
  where the constant $C_R$ depends on the $\test_{xxx}$. Collecting
  these bounds we get
    
  \begin{multline}
    \label{eq:l2_contr1}
    (1-\Dt C_R)\int_\R \left(w^{\ell+1}- w^\ell\right)^2 \varphi\, dx   \\
    \le C_{R} \lambda^2 \max\seq{\norm{w^\ell}_{2,\varphi}^2,
      \norm{w^{\ell-1}}_{2,\varphi}^2} \int_\R
    \left(w^\ell-w^{\ell-1}\right)^2\varphi\,dx
  \end{multline}
  We can always assume that $C_R\Dt<1/2$, thus
  \begin{equation}
    \label{iteration diff}
    \begin{aligned}
      &\norm{\left(w^{\ell+1}-w^\ell\right)}_{2,\varphi}^2 \\
      &\qquad\qquad \le 2C_{R} \lambda^2
      \max\seq{\norm{w^\ell}_{2,\varphi}^2,
        \norm{w^{\ell-1}}_{2,\varphi}^2}\norm{
        \left(w^{\ell}-w^{\ell-1}\right)}_{2,\varphi}^2.
    \end{aligned}
  \end{equation}
  For $w^1$, setting $v=w^1$ in \eqref{eq:iteration scheme}, we have
  \begin{align*}
    \int_\R \left(w^1\right)^2\varphi\,dx +& \Dt \int_\R w^1_x \left(\varphi
      w^1\right)_{xx}\,dx \\
    &= \int_\R
    w^1 \left(u^n - \Dt u^n u^n_x\right)  \varphi\,dx\\
    &\le \frac12 \int_\R \left(w^1\right)^2 \varphi\,dx +
    \frac{1}{2} \int_\R \left(u^n - \Dt u^n u^n_x\right)^2\varphi \,dx\\
    &\le \frac12 \int_\R \left(w^1\right)^2 \varphi\,dx + \int_\R
    \left(u^n\right)^2\varphi\,dx + \Dt^2 \int_\R \left(u^n
      u^n_x\right)^2\varphi\,dx.
  \end{align*}
  Therefore, using the inverse inequality \eqref{eq:inverse_inequal}
  and the identity \eqref{eq:identity}, we have
  \begin{align*}
    \frac12 \norm{w^1}_{2,\varphi}^2 &\le \norm{u^n}_{2,\varphi}^2 +
    \frac{\Dt}2 \int_\R \left(w^1\right)^2 \varphi_{xxx}\,dx + \Dt^2
    \norm{u^n_x}_{L^\infty(\R)}^2
    \int_\R  \left(u^n\right)^2  \varphi\,dx \\
    &\le \norm{u^n}_{2,\varphi}^2 + \frac{C_R\Dt}2
    \norm{w^1}_{2,\varphi}^2 +
    C_R\frac{\Dt^2}{\Dx^3}\norm{u^n}_{2,\varphi}^4,
  \end{align*}
  and thus
  \begin{equation}
    \label{eq:w1start}
    \norm{w^1}_{2,\varphi}^2\le 4\left(1+C_R\lambda^2\norm{u^n}_{2,\varphi}^2\right)\norm{u^n}_{2,\varphi}^2.
  \end{equation}
  Then we claim that the following holds for $\ell\ge 1$
  \begin{subequations}
    \label{eq:iteration}
    \begin{align}
      \norm{w^{\ell+1}-w^\ell}_{2,\varphi} &\le
      L\norm{w^{\ell}-w^{\ell-1}}_{2,\varphi}, \label{eq:iterationa}
      \\
      \norm{w^\ell}_{2,\varphi} &\le K
      \norm{u^n}_{2,\varphi},\label{eq:iterationb}\\
      \norm{w^1}_{2,\varphi} &\le 5\norm{u^n}_{2,\varphi},
      \label{eq:iterationc}
    \end{align}
  \end{subequations}
  for $\ell=1,2,3,\ldots$.  To prove these claims, we argue by
  induction. Setting $\ell=1$ in \eqref{iteration diff} and using
  \eqref{eq:w1start} gives
  \begin{align*}
   \norm{w^2-w^1}_{2,\varphi} &\le 2 \sqrt{C_R}\lambda
    \max\seq{\norm{w^1}_{2,\varphi},\norm{u^n}_{2,\varphi}}
    \norm{w^1-u^n}_{2,\varphi}\\
    &\le 
    \frac{2\sqrt{C_R}L}{2\sqrt{2}\sqrt{C_R}
      K\norm{u^n}_{2,\varphi}}\norm{u^n}_{2,\varphi} \\
    &\qquad \times
    4\left(1+
      \frac{\sqrt{C_R}L}{\sqrt{C_R}2\sqrt{2}K\norm{u^n}_{2,\varphi}}
      \norm{u^n}_{2,\varphi} 
    \right)
    \norm{w^1-u^n}_{2,\varphi}\\
    &\le
    \frac{4L}{7\sqrt{2}}\left(1+\frac{L}{14\sqrt{2}}\right)\norm{w^1-u^n}_{2,\varphi}
    \\
    &\le L \frac{4}{7\sqrt{2}}\left(1+\frac{1}{14\sqrt{2}}\right)
    \norm{w^1-u^n}_{2,\varphi}\\
    &\le 0.85 L \norm{w^1-u^n}_{2,\varphi},
  \end{align*}
  which shows \eqref{eq:iterationa} for $\ell=1$.  To show
  \eqref{eq:iterationc} note that
  \begin{equation*}
    4\left(1+\sqrt{C_R}\lambda\norm{u^n}_{2,\varphi}\right)\le
    4\left(1+\frac{1}{14\sqrt{2}}\right)<5.
  \end{equation*}
  Next assume that \eqref{eq:iterationa} and \eqref{eq:iterationb}
  hold for $\ell=1,\ldots,m$, 
  then
  \begin{align*}
    \norm{w^{m+1}}_{2,\varphi}&\le \sum_{\ell=0}^{m}
    \norm{w^{\ell+1}-w^\ell}_{2,\varphi} + \norm{w^0}_{2,\varphi}\\
    &\le \norm{w^1-w^0}_{2,\varphi}\sum_{\ell=0}^m L^\ell +
    \norm{w^0}_{2,\varphi}\\
    &\le 6 \norm{u^n}_{2,\varphi} \frac{1}{1-L} +
    \norm{u^n}_{2,\varphi}\\
    &=\frac{7-L}{1-L}\norm{u^n}_{2,\varphi}=K\norm{u^n}_{2,\varphi}.
  \end{align*}
  Hence, \eqref{eq:iterationb} holds for all $\ell$. Using
  \eqref{iteration diff}, this implies that \eqref{eq:iterationa}
  holds as well. Using \eqref{iteration diff}, one can show that
  $\{w^{\ell}\}$ is Cauchy, hence $\{w^{\ell}\}$ converges. This
  completes the proof.
\end{proof}

\begin{remark}
  Note that, we aim to prove that the iteration scheme
  \eqref{eq:iteration scheme} converges for all times $t_n = n
  \Dt$. We have already shown in previous section that the iteration
  scheme converges for one time step. However, we had to impose a CFL
  condition where the ratio between temporal and spatial mesh sizes
  must be smaller than an upper bound that depends on the computed
  solution at that time, i.e., $u^n$. Having said this, since we want
  the CFL-condition to only depend on the initial data
  $u_0$, we have to derive local a-priori bounds for the computed
  solution $u^n$. This will be done in the next section (cf.~bound
  \eqref{eq:l2_stability}).  This bound finally implies that the
  iteration scheme \eqref{eq:iteration scheme} converges for
  sufficiently small $\Dt$.

\end{remark}

\section{ Convergence}
\label{sec:convergence}
As we mentioned earlier, the convergence analysis  exploits the
fact that the solution of the KdV equation possesses an inherent
smoothing effect due to its dispersive character. In particular, we
need $H^1_{\mathrm{loc}}(\R)$ estimate of the approximate solution
generated by the scheme \eqref{eq:scheme}. We proceed with the
following Lemma.

\begin{lemma}
  \label{lemma2}
  We assume that $K$ and $L$ are given as in the hypothesis of Lemma
  \ref{lemma1}. We assume that the initial data $u_0\in L^2(\R)$. Let
  $u^n$ be the solution of the scheme \eqref{eq:scheme}. Then there
  exists a finite time $T$ and a constant $C$, depending only on
  $\norm{u_0}_{L^2(\R)}$, such that for all $n$ satisfying $n \Dt \le T$,
  the following estimate holds
  \begin{equation}
    \label{eq:l2_stability}
    \norm{u_n}_{L^2(\R)} \le C\left(\norm{u_0}_{L^2(\R)}\right)
  \end{equation}
  provided the following assumption holds
  \begin{equation}
    \label{CFL1}
    \lambda\le \frac{L}{\sqrt{C_R} 2\sqrt{2} K \sqrt{y_T}}
  \end{equation}
  for some $y_T$ which
  depends only on $\norm{u_0}_{L^2(\R)}$.  Furthermore, the approximation
  $u^n$ satisfies the following $H^1$ estimate
  \begin{equation}
    \label{eq:H1_estimate}
    \Dt\sum_{ n\Dt\le T} \norm{u_x^{n+1}}_{L^2([-R,R])}^2\le
    C\left(\norm{u^0}_{L^2(\R)},R\right), \qquad \text{for $n\Dt<T$,}
  \end{equation}
  where the constant $C$ depends only on $R$ and $\norm{u_0}_{L^2(\R)}$.
\end{lemma} 
\begin{proof}
We choose $v=u^{n+1}$ in \eqref{eq:scheme} to
obtain
\begin{equation*}
  \int_\R (u^{n+1})^2 \varphi \,dx \,+\, \Dt \int_\R u^{n+1}_x (\varphi
  u^{n+1})_{xx}\,dx =\int_\R u^n \varphi u^{n+1}\,dx \,-\, \Dt\int_\R
  (u^{n+1})^2u^{n+1}_x\varphi \,dx.
\end{equation*}
Using Cauchy-Schwartz inequality and the identity \eqref{eq:identity},
we have
\begin{multline}
  \label{eq:estimate1}
  \frac12\int _\R(u^{n+1})^2 \varphi dx + \Dt \int_\R u^{n+1}_x (\varphi
  u^{n+1})_{xx}dx\\
  \le \frac12 \int _\R (u^n)^2 \varphi dx +\frac{\Dt}{3} \int_\R (u^{n+1})^3
  \varphi_x dx.
\end{multline}
Next we estimate $\frac{\Dt}{3} \int_\R (u^{n+1})^3 \varphi_x dx$. To do
that, we make an use of the following identity
\begin{equation*}
  \sup_{x\in \R}v^2(x)\le \frac12 \int_\R |v(x)| |v_x(x)|dx
\end{equation*}
valid for $v\in H^1(\R)$.  Taking $v=u\sqrt{\varphi_x}$ in
we obtain
\begin{align*}
  \sup_x \abs{u\sqrt{\varphi_x}} & \le \frac1{\sqrt{2}} \Bigl( \int
    \abs{u\sqrt{\varphi_x}} \abs{(u\sqrt{\varphi_x})_x}\, dx
  \Bigr)^{\frac12}\\ 
  & \le \frac1{\sqrt{2}} \Bigl( \int_\R \abs{u u_x
      \varphi_x\,dx}\Bigr)^{\frac12} \,
  + \, \frac12  \Bigl( \int_\R \abs{u^2 \varphi_{xx}}\,dx\Bigr)^{\frac12}\\
  &\le \frac1{\sqrt{2}} \Bigl( \int_\R u_x^2 \varphi_x \,dx
  \Bigr)^{\frac14} \Bigl(\int_\R u^2 \varphi_x \,dx\Bigr)^{\frac14} +
  \frac12 \Bigr( \int_\R u^2 \abs{\varphi_{xx}}\,dx\Bigl)^{\frac12}
\end{align*}
by the Cauchy-Schwartz inequality.  Therefore
\begin{align*}
  \int_\R u^3\varphi_x \,dx   
  &\le\Bigl(\sup \abs{u\sqrt{\varphi_x}}\Bigr)  
  \int_\R u^2\sqrt{\phi_x} \,dx
  \\ 
  &\le \frac{1}{\sqrt{2}} \Bigl( \int_\R u_x^2 \varphi_x \,dx
  \Bigr)^{\frac14} \Bigl(\int_\R u^2 \varphi_x \,dx\Bigr)^{\frac14}
  \Bigl(\int_\R u^2\sqrt{\phi_x} \,dx \Bigr)
  \\
  &\qquad + \frac12 \Bigl(\int_\R u^2\sqrt{\phi_x} \,dx \Bigr)
  \Bigl(\int_\R u^2 \abs{\varphi_{xx}}\,dx\Bigr)^{\frac12}.
\end{align*}
Applying  Young's inequality $ab \le \tfrac{1}{4}a^4+
\tfrac34 b^{4/3}$ for non-negative numbers $a$ and $b$, we
obtain 
\begin{multline}
  \label{eq:inequalitya}
  \frac13\int_\R u^3\varphi_x \,dx \le \frac1{12 \sqrt{2}} \int_\R u_x^2
  \varphi_x \,dx +\frac1{4\sqrt{2}} \Bigl(\int_\R u^2 \varphi_x
  \,dx\Bigr)^{\frac13} \Bigl(\int_\R u^2\sqrt{\phi_x} \,dx \Bigr)^\frac43
  \\
  + \frac16 \Bigl(\int_\R u^2\sqrt{\phi_x} \,dx \Bigr) \Bigl( \int_\R u^2
  \abs{\varphi_{xx}}\,dx\Bigr)^{\frac12} .
\end{multline}
  Replacing the last term of \eqref{eq:estimate1} by the above
inequality \eqref{eq:inequalitya}, and using the identity
\eqref{eq:identity} for the second term of \eqref{eq:estimate1} gives
\begin{equation}
  \label{Nrgy1}
  \begin{multlined}
    \!\!\!\!\!\!\!\!\!
    \frac12\int_\R \left(u^{n+1}\right)^2 \varphi \,dx  +   \Bigl(\frac32
    - \frac1{12\sqrt{2}}\Bigr)\Dt \int_\R \left(u^{n+1}_x
    \right)^2\varphi_{x}\,dx
    \\
   \le \frac12 \int_\R \left(u^n\right)^2 \varphi \,dx
    +\frac{\Dt}{4\sqrt{2}} \Bigl(\int_\R \left(u^{n+1}\right)^2 
    \varphi_x \,dx\Bigr)^{\frac13} 
    \Bigl(\int_\R \left(u^{n+1}\right)^2\sqrt{\phi_x} \,dx\Bigr)^\frac43   
    \\
    +\frac{\Dt}{6}\Bigl(\int_\R \left(u^{n+1}\right)^2\sqrt{\phi_x}\, dx
    \Bigr) \Bigl( \int_\R \left(u^{n+1}\right)^2\abs{\varphi_{xx}}\,dx\Bigr)^\frac12\\
    + \frac{\Dt}2 \int_\R \left(u^{n+1}\right)^2
    \abs{\varphi_{xxx}}\,dx.
 \end{multlined}
\end{equation}
As the derivatives of $\varphi$ are bounded by the constant $C_R$, the
derivatives of $\varphi^{(j)}(x)\le \varphi(x)$ for $j=1,2,3$. Thus, from
\eqref{Nrgy1}, we obtain
\begin{equation}
  \label{Nrgy2}
  \begin{aligned}
    \int_\R \left(u^{n+1}\right)^2 \varphi\, dx \le  \int_\R
    \left(u^n\right)^2
    \varphi\,& dx 
    + \Dt \,C_R\Bigl[\Bigl(\int_\R \left(u^{n+1}\right)^2
    \varphi\,dx\Bigr)^{\frac53}
    \\
    & + \Bigl(\int_\R \left(u^{n+1}\right)^2 \varphi\,
        dx\Bigr)^{\frac32} + \Bigl(\int_\R \left(u^{n+1}\right)^2 \varphi\,
        dx\Bigr)\Bigr].
  \end{aligned}
\end{equation}
We have ignored the second term in the inequality \eqref{Nrgy1}
since the coefficient $\Dt(\tfrac32 - \tfrac1{12\sqrt{2}})$ is positive.
Setting $a_n=\int (u^n)^2 \varphi\, dx$ in \eqref{Nrgy2} gives
\begin{equation}
  \label{ode}
  a_{n+1} \le a_{n} +\Dt \,  f(a_{n+1})
\end{equation}
where the function $f$ is given by
\begin{align*}
  f(a)=C_R \, \left[ a^{\frac53} + a^{\frac32} +a\right].
\end{align*}
Therefore, $\{a_n\}$ solves the implicit Backward Euler method for the
following differential inequality
$$ \frac{da}{dt}\le f(a).$$
Thus we consider the following ordinary differential equation
\begin{align*}
  \begin{cases}
    \frac{dy}{dt}=f(K^2 y), &t>0,\\ y(0)=a_0.
  \end{cases}
\end{align*}
Since the function $f$ is locally Lipschitz continuous for positive
arguments, this differential equation has a unique solution which
blows up at some finite time, say at $t=T^{\infty}$. We choose
$T=T^{\infty}/2$. Also, note that the solution $y(t)$ of the above
differential equation is \emph{strictly-increasing} and \emph{convex}.
Next we compare the solution of this ODE with \eqref{ode} under the
assumption that \eqref{CFL1} holds.

Next we claim that $a_n\le y(t_n)$ for all $n\geq 0$.  We argue by
induction. Since $y(0)=a_0$, the claim follows for $n=0$.
We assume that the claim holds for $n=0,1,2,...,m$. As $0<a_m\le
y(T)$, \eqref{CFL1} implies that $\lambda$ satisfies the CFL
condition \eqref{eq:cfl}. So, from Lemma~\ref{lemma1}, we have
$a_{m+1}\le K^2 a_m$.

Then, using the convexity of $f$ we have
\begin{align*}
  a_{m+1}& \le a_{m} +\Dt \, f(K^2 a_m)\\
  &\le y(t_m)+\Dt f(K^2 y(t_m))\\
  &\le y(t_m)+\Dt \frac{dy}{dt}\bigm|_{t=t_m} \le
  y(t_{m+1}).
\end{align*}
This proves the claim. Therefore, as $\varphi\geq 1$, we have the
required $L^2$-stability estimate
$$\norm{u^n}_{L^2(\R)}\le \sqrt{y(T)}\le C\left(\norm{u^0}_{L^2(\R)},R\right).$$
Therefore, summing \eqref{Nrgy1} over $n$, we obtain
\begin{align*}
  \Dt\sum_{n\Dt\le T}\int_{-R}^R \abs{u^{n+1}_x}^2dx \le C(R,
  \norm{u_0}_{L^2(\R)}).
\end{align*}
This proves \eqref{eq:H1_estimate} and
completes the proof of Lemma~\ref{lemma2}.
\end{proof}

\subsection{Bounds on temporal derivative}
Next, we estimate the temporal derivative of the approximate solution.
In doing so, we need the following lemma which some bounds on
a weighted-$L^2$ projection on the space of $S_{\Dx}$ corresponding
to the weight function $\varphi$.
\begin{lemma}
  Let $\psi \in C_c^{\infty}(-R,R)$. Then there exists a projection
  $P:C^\infty_c(-R,R)\to S_\Dx\cap C_c(-R,R)$ such that
  \begin{align*}
    \int_\R u P(\psi) \varphi\,dx=\int_\R u \psi \varphi \,dx\qquad
    \text{for all $u\in S_\Dx$.}
  \end{align*}
  In addition, $P$ satisfies the following
  bounds
  \begin{equation}
    \label{eq:projection}
    \begin{aligned}
      \begin{cases}
        \norm{P(\psi)}_{L^2(\R)}\le C\norm{\psi}_{L^2(\R)},\\
        \norm{P(\psi)}_{H^1(\R)}\le C\norm{\psi}_{H^1(\R)},\\
        \norm{P(\psi)}_{H^2(\R)}\le C\norm{\psi}_{H^2(\R)}
      \end{cases}
    \end{aligned}
  \end{equation}
  where the constant $C$ is independent of $\Dx$.
\end{lemma}
\begin{proof}
  This proof is an easy adaptation of the classical $L^2$ projection
  results, see the monograph of Ciarlet \cite{ciarlet}.
\end{proof}
\begin{lemma}
  \label{lemma3}
  Let $\{u_n\}$ be the solution of the scheme \eqref{eq:scheme}. We
  also assume that the hypothesis of Lemma~\ref{lemma2} hold. Then
  the following estimate holds
  \begin{equation}
    \label{eq:time estimate}
    \norm{D_t^+(u^n\varphi)}_{H^{-2}([-R,R])}\le C(\norm{u_0}_{L^2(\R)},R)\,
    \left(\norm{u^{n+1}_x}_{L^2([-R,R])}+1\right), 
  \end{equation}
  where $D_t^+u^n$ is the forward time difference given by
  \begin{equation*}
    D_t^+u^n=\frac{u^{n+1}-u^n}{\Dt}.   
  \end{equation*}
\end{lemma}
\begin{proof}
  Using the definition of $D_t^+u^n$, we rewrite the scheme
  \eqref{eq:scheme} as
  \begin{equation}
    \label{eq:schemea}
    \left(D_t^+u^n, \varphi v\right) - 
    \left(\frac{(u^{n+1})^2}{2}, (\varphi v)_x\right)
    +  \left(u^{n+1}_x,(\varphi v)_{xx}\right)  = 0,
  \end{equation}
  which holds for all $v\in S_{\Dx}$.  Let $\psi\in C^\infty_c(-R,R)$
  and choose $v=P(\psi)$ in
  \eqref{eq:schemea}, to  obtain
  \begin{equation*} 
    (D_t^+u^n, \varphi P(\psi)) + \left(u^{n+1}u^{n+1}_x, \varphi
      P(\psi)\right) + (u^{n+1}_x,(\varphi P(\psi))_{xx}) = 0.
  \end{equation*}
  The second and third terms of the above identity 
  can be estimated as follows. Using the bounds
  \eqref{eq:projection} and the Sobolev inequality we get
  \begin{align*}
    -\int_\R & \left((u^{n+1})^2\right)_x \varphi P(\psi)\, dx = \int_\R
    \left(u^{n+1}\right)^2 \varphi_x P(\psi) \,dx + \int_\R
    \left(u^{n+1}\right)^2\varphi P(\psi)_x \,dx
    \\
    &\le \left(\norm{P(\psi)}_{L^\infty([-R,R])}  +
    \norm{P(\psi)_x}_{L^\infty([-R,R])}(2R+1)\right) \int_{-R}^R\left(u^{n+1}\right)^2\,dx
    \\
    & \quad \le \left(\norm{P(\psi)}_{H^1([-R,R])} +
      \norm{P(\psi)_x}_{H^1([-R,R])}(2R+1)\right) \norm{u^{n+1}}_{L^2(\R)}^2
    \\
    & \qquad \le C\left(\norm{u_0}_{L^2(\R)},R\right) \norm{\psi}_{H^2([-R,R])},
  \end{align*}
  and
  \begin{align*}
    -\int_\R u^{n+1}_x(\varphi P(\psi))_{xx} \,dx &\le \norm{u^{n+1}_x}_{L^2([-R,R])}
    \norm{(\varphi P(\psi))_{xx}}_{L^2(\R)}
    \\ 
    &\le C(\norm{u_0}_{L^2(\R)},R)\, \norm{u^{n+1}_x}_{L^2([-R,R])}\,
    \norm{\psi}_{H^2(\R)}.
  \end{align*}
  Therefore
  \begin{align*}
    \Bigl|\int_\R D_t^+u^n\varphi \psi \,dx\Bigr| &= \Bigl|\int_\R
    D_t^+u^n\varphi P(\psi)\,dx \Bigr|
    \\
    &\le
    C(\norm{u_0}_{L^2(\R)},R) \left(\norm{u^{n+1}_x}_{L^2([-R,R])}+1\right)
    \norm{\psi}_{H^2(\R)},
  \end{align*}
  which completes the proof.
\end{proof}
Before stating the theorem of convergence, we define the weak solution of the Cauchy problem \eqref{eq:main} as follows.
\begin{definition}
  Let $Q$ be a given positive number. Then $u\in L^2(0,T; H^1(-Q,Q))$ is said
  to be a weak solution of \eqref{eq:main} in the interval $(-Q,Q)$ if
  \begin{equation}\label{weak solution}
    \int_0^T \int_{-\infty}^{\infty} \Bigl(\phi_tu + \phi_x
    \frac{u^2}{2} - \phi_{xx}u_x \Bigr)\,dxdt + \int_{-\infty}^{\infty}
    \phi(x,0)u_0(x)\,dx =0.
  \end{equation}
  for all $\phi\in C^{\infty}_c\left((-Q,Q)\times[0,T)\right)$.
\end{definition}

Next we define the approximation $u^{\Dx}$
as,
\begin{equation}
  \label{eq:approx}
  u^{\Dx}(x,t)=u^n(x) \, + \, (t-t_n) D_t^+u^n, \, \, \, t_n\le t< t_{n+1}.
\end{equation}
Then we have the following theorem for convergence.
\begin{theorem}
  \label{theo:main}
  Let $\{u^n\}_{n\in\N}$ be a sequence of functions defined by the
  scheme \eqref{eq:scheme}, and assume that $\norm{u_0}_{L^2(\R)}$ is
  finite. Assume furthermore that $\Dt= \mathcal{O}(\Dx^2)$, then
  there exists a constant $C$ (depends only on $R$ and $\norm{u_0}_{L^2(\R)}$)
  such that
  \begin{align}
    \label{eq:aaa}\norm{u^{\Dx}}_{L^\infty(0,T;L^2([-R,R]))}\le C(R,\norm{u_0}_{L^2(\R)} ),\\
    \label{eq:bbb}\norm{u^{\Dx}}_{L^2(0,T;H^1([-R,R]))}\le C(R,\norm{u_0}_{L^2(\R)}),\\
    \label{eq:ccc} \norm{\partial_t
      (u^{\Dx}\varphi)}_{L^2(0,T;H^{-2}([-R,R]))}\le C(R,\norm{u_0}_{L^2(\R)})
  \end{align}
  where $u^{\Dx}$ is given by \eqref{eq:approx}. Moreover, there
  exists a sequence of $\{\Dx_j\}_{j=1}^{\infty}$ with
  $\lim_{j\rightarrow \infty}$ and a function $u\in
  L^2(0,T;L^2([-R,R]))$ such that
  \begin{equation}\label{eq:convergence}
    u^{\Dx_j}\rightarrow u \text{  strongly in  } L^2(0,T;L^2([-R,R])),
  \end{equation}
  as $j$ goes to infinity. The function $u$ is a weak solution of
  the Cauchy problem \eqref{eq:main}, that is, it satisfies
 \eqref{weak solution} with $Q=R-1$.
\end{theorem}
\begin{proof}
  We write the approximation $u^{\Dx}$ as, for $t_n\le t<t_{n+1}$
  \begin{align*}
    u^{\Dx}(x,t)=(1-\alpha_n(t)) u^n(x) \, + \, \alpha_n(t)
    u^{n+1}(x),
  \end{align*}
  where $\alpha_n(t)=(t-t_n)/{\Dt}\in [0,1]$. Therefore, we have
  \begin{align*}
    \norm{u^{\Dx}}_{L^2(\R)}\le \norm{u^n}_{L^2(\R)}+ \norm{u^{n+1}}_{L^2(\R)}.
  \end{align*}
  Thus, using \eqref{eq:l2_stability}, we conclude that \eqref{eq:aaa}
  holds.
 
  To prove \eqref{eq:bbb}, we calculate, for $t\in [t_n, t_{n+1})$
  \begin{align*}
    \norm{u^{\Dx}_x}_{L^2([-R,R])}^2\le 2(1-\alpha_n(t))^2
    \norm{u^n_x}_{L^2([-R,R])}^2\, + \, 2\alpha_n(t)^2
    \norm{u^{n+1}_x}_{L^2([-R,R])}^2.
  \end{align*}
  Thus,
  \begin{align*}
    \int_0^T \norm{u^{\Dx}_x}_{L^2([-R,R])}^2 dt&\le  2 \int_0^T (1-\alpha_n(t))^2 \norm{u^n_x}_{L^2([-R,R])}^2 \,dt\, \\
    & \qquad \qquad + \, 2 \int_0^T\alpha_n(t)^2 \norm{u^{n+1}_x}_{L^2([-R,R])}^2 \, dt\\
    &= 2 \sum_{n=0}^{N-1}\norm{u^n_x}_{L^2([-R,R])} ^2\int_{t_n}^{t_{n+1}}(1-\alpha_n(t))^2dt\\
    &\hspace{2cm} +\,2 \sum_{n=0}^{N-1}\norm{u^{n+1}_x}_{L^2([-R,R])}^2 \int_{t_n}^{t_{n+1}}\alpha_n(t)^2 \, dt \\
    &\le\Dt \sum_{n=0}^{N-1}\norm{u^n_x}_{L^2([-R,R])} \, +\Dt \sum_{n=0}^{N-1}\norm{u^{n+1}_x}_{L^2([-R,R])}\\
    &\le \Dt \norm{u^0_x}^2_{L^2([-R,R])} \, + \, 2 \Dt \sum_{n=0}^{N-1}\norm{u^{n+1}_x}_{L^2([-R,R])}\\
    & \le C(\norm{u_0}_{L^2(\R)}, R)
  \end{align*}
  where $N$ satisfies $N\Dt=T$. Here we have used the inverse
  inequality \eqref{eq:inverse_inequal} for the first term and the
  estimate \eqref{eq:H1_estimate} for the second term. This concludes
  the proof of \eqref{eq:bbb}.
  
  Next we prove \eqref{eq:ccc}. We first note that, for $t \in [t_n,
  t_{n+1})$
  \begin{equation*}
    \partial_t u^{\Dx}(x,t)=D_t^+u^n.
  \end{equation*}
  Thus, using Lemma \ref{lemma3} and Lemma \ref{lemma2}, we have
  \begin{align*}
    &\int_0^T\norm{\partial_t u^{\Dx}}^2_{H^{-2}([-R,R])} dt \le
    C\int^T_0 \norm{ u^{n+1}_x}^2_{L^{2}([-R,R])} \, dt\\
    & \le C\, \Dt \sum_{n=0}^{N-1} \norm{u^{n+1}_x}^2_{L^2([-R,R])} \le
    C(\norm{u_0}_{L^2(\R)},R).
  \end{align*}
  This shows that \eqref{eq:ccc} holds.
  
  Since $\varphi$ is a positive and bounded smooth function, using
  \eqref{eq:aaa}, \eqref{eq:bbb} we have
  \begin{subequations}
    \label{eq:phi_estimate}
    \begin{align}
      \label{eq:phi1}\norm{\test u^{\Dx}}_{L^\infty(0,T;L^2([-R,R]))}\le
      C(\norm{u_0}_{L^2(\R)},R),\\
      \label{eq:phi2}\norm{\test u^{\Dx}}_{L^2(0,T;H^1([-R,R]))}\le
      C(\norm{u_0}_{L^2(\R)},R).
    \end{align}
  \end{subequations}
  Using \eqref{eq:phi_estimate} and \eqref{eq:ccc} we can apply the
  Aubin-Simon compactness lemma (see \cite{HKR1}) applied to the set
  $\seq{\test u^\Dx}_{\Dx>0}$ to conclude that there exist a sequence
  $\seq{\Dx_j}_{j\in\N}$ such that $\Dx_j\to 0$, and a function
  $\tilde{u}$ such that
  \begin{equation}\label{eq:phi_convergence}
    u^{\Dx_j}\varphi\rightarrow \tilde{u}\qquad \text{strongly in  
      $L^2(0,T;L^2([-R,R]))$,} 
  \end{equation}
  as $j$ goes to infinity.  As $\test\ge 1$,
  \eqref{eq:phi_convergence} implies that there exists a $u$ such that
  \eqref{eq:convergence} holds.
  
  This strong convergence allows passage to the limit in
  nonlinearity. However, it remains to prove that u is a weak solution
  of \eqref{eq:main}.  In what follows, we will consider the standard
  $L^2$-projection of a function $\psi$ with $k+1$ continuous
  derivatives into space $S_{\Dx}$, denoted by $\mathcal{P}$, i.e.,
  \begin{align*}
    \int_{\R} \left(\mathcal{P} \psi(x) - \psi(x) \right) v(x) =0,
    \quad \forall v \in S_{\Dx}.
  \end{align*}
  For the projection mentioned above we have that (for a proof,
  see the monograph of Ciarlet \cite{ciarlet})
  \begin{align*}
    \norm{\psi(x) - \mathcal{P} \psi(x)}_{H^k(\R)} \le C \Dx
    \norm{\psi}_{H^{k+1}(\R)},
  \end{align*}
  where $C$ is a constant independent of $\Dx$.

  We also need the following inequality:
  \begin{equation} \label{sobolev ineq} \norm{u^n}_{L^{\infty}[-R+1,
      R-1]} \le C(R) \norm{u^n}_{H^1(-R,R)}
  \end{equation}
  where $C_R$ is some positive constant depends only on $R$. To show
  this inequality, we consider the the smooth function $\eta$ such
  that $\eta=1$ on $[-R+1, -R-1]$ and $\eta=0$ on the set $\seq{\abs{x}>
  R-\frac12}$. Then, it is easy to see that
  \begin{equation*}
    \abs{u^n(x) \eta(x)} \le \left(\norm{\eta}_{L^\infty(\R)}+
      \norm{\eta_x}_{L^\infty(\R)}\right) (2R)^{1/2}\norm{u^n}_{H^1([-R,R])}.
  \end{equation*}
  As $\eta =1$ on $[-R+1, R-1]$, we conclude that \eqref{sobolev ineq}
  holds.

  We first show that
  \begin{equation}
    \begin{aligned}
      \label{eq:clm}
      \int_0^T \int_{\R} u^{\Dx}_t \varphi v - \frac{(u^{\Dx})^2}{2}
      (\varphi v)_x + (u^{\Dx})_x (\varphi v)_{xx}
      \,dx\,dt=\mathcal{O}(\Dx),
    \end{aligned}
  \end{equation}
  for any test function $v \in C_c^{\infty} \left((-R+1,R-1) \times
    [0,T)\right)$, where $\varphi$ is specified in the
  beginning of Section~\ref{sec:scheme}.
 
  Let $v^{\Dx}=\mathcal{P}v$, then from the definition of $u_{\Dx}$
  (c.f.~\eqref{eq:approx}), it is evident that
  \begin{align*}
    &\int_0^T \int_{\R}  u^{\Dx}_t \varphi v - \frac{(u^{\Dx})^2}{2}
    (\varphi v)_x + (u^{\Dx})_x (\varphi v)_{xx}\,dxdt \\ 
    & \quad = \underbrace{\sum_{n} \int_{\R} \int_{t_n}^{t_{n+1}}
      D^{+}_{t} u^n \varphi v^{\Dx} - \frac{(u^{n+1})^2}{2} (\varphi
      v^{\Dx})_x
      + (u^{n+1})_x (\varphi v^{\Dx})_{xx} \,dtdx}_{=0 \, \text{by}\,
      \eqref{eq:scheme}} \\ 
    & \qquad + \sum_{n} \int_{\R} \int_{t_n}^{t_{n+1}}
    \underbrace{D^{+}_{t} u^n \left(\varphi v -\varphi v^{\Dx}
      \right)}_{\mathcal{E}^{1,n}_{\Dx}}
    -  \underbrace{\frac{(u^{n+1})^2}{2} \left( \varphi v - \varphi
        v^{\Dx} \right)_x}_{\mathcal{E}^{2,n}_{\Dx}}\,dtdx \\ 
    & \qquad  +\sum_{n} \int_{\R} \int_{t_n}^{t_{n+1}}
    \underbrace{ (u^{n+1})_x \left( \varphi v - \varphi
        v^{\Dx}\right)_{xx}}_{\mathcal{E}^{3,n}_{\Dx}}
    + \underbrace{ \left(u^{\Dx}- u^{n+1} \right)_x \left( \varphi v
      \right)_{xx}}_{\mathcal{E}^{4,n}_{\Dx}} \,dtdx \\ 
    & \qquad + \sum_{n} \int_{\R} \int_{t_n}^{t_{n+1}}
    \underbrace{\left( - \frac{(u^{\Dx})^2}{2} + \frac{(u^{n+1})^2}{2}
      \right)(\varphi v)_x }_{\mathcal{E}^{5,n}_{\Dx}}\,dtdx
  \end{align*}
  We proceed with 
  \begin{align*}
   &\Bigl| \sum_{n} \int_{\R} \int_{t_n}^{t_{n+1}}\!\!\!\!\! \mathcal{E}^{1,n}_{\Dx}
    \,dx\,dt\Bigr| = \Bigl|\sum_{n} \int_{\R} \int_{t_n}^{t_{n+1}} D^{+}_{t} u^n
    \left(\varphi v -\varphi v^{\Dx} \right)\,dtdx \Bigr|
    \\ 
    &\le C(R)  \sum_{n}  \int_{t_n}^{t_{n+1}} \norm{D^{+}_t \left( u^n
        \varphi\right)}_{H^{-2}([-R,R])} \norm{v -
      v^{\Dx}}_{H^2([-R+1,R-1])} \,dt\\ 
    &\le \Dx \,C(\norm{u_0}_{L^2(\R)}, R)
    \norm{v}_{L^2\left((0,T);H^3([-R+1,R-1])\right)} \to 0,
    \ \text{as} \, \Dx \downarrow 0.
  \end{align*}
  Next, using \eqref{sobolev ineq}, we obtain
  \begin{align*}
    \Bigl|&\sum_{n} \int_{\R} \int_{t_n}^{t_{n+1}} \mathcal{E}^{2,n}_{\Dx}
    \,dtdx\Bigr| 
    = \Bigl|\sum_{n} \int_{\R} \int_{t_n}^{t_{n+1}}
    \left(\frac{u^{n+1}}{2} \right)^2 \left( \varphi v - \varphi
      v^{\Dx} \right)_x \,dtdx\Bigr|
    \\
    & \quad = \Bigl|\sum_{n} \int_{-R+1}^{R-1} \int_{t_n}^{t_{n+1}}
    \left(\frac{u^{n+1}}{2} \right)^2  \left( \varphi \left( v -
        v^{\Dx} \right)_x +  \left( v-v^{\Dx} \right) \varphi _x
    \right)\,dtdx \Bigr|\\ 
    &\quad \le  C(R)\left(\sum_{n}\int_{t_n}^{t_{n+1}}\!\!\!\!\!
    \norm{u^{n+1}}_{L^{\infty}([-R+1,R-1])} \int_{-R+1}^{R-1}
    \abs{u^{n+1}}  \abs{ \left( v - v^{\Dx} \right)_x} \,dtdx\right.\\
   &\hspace{2cm}+ \left.\sum_{n}\int_{t_n}^{t_{n+1}}\!\!\!\!\!
          \norm{u^{n+1}}_{L^{\infty}([-R+1,R-1])} \int_{-R+1}^{R-1}
                 \abs{u^{n+1}}  \abs{ \left( v - v^{\Dx} \right)} \,dtdx\right)\\
    & \quad \le C(R)\,\sum_{n}\int_{t_n}^{t_{n+1}}
    \norm{u^{n+1}}^2_{H^1([-R,R])}\, \norm{v-v^{\Dx}}_{H^1([-R+1,R-1])}dt\\ 
    & \quad \le C(\norm{u_0}_{L^2(\R)}, R) \,\Dx
    \norm{v}_{L^{\infty}((0,T);H^2([-R+1,R-1]))} \to 0, \ \text{as}
    \ \Dx \downarrow 0.
  \end{align*}
  Using the Cauchy-Schwartz inequality
  \begin{align*}
   & \Bigl|\sum_{n} \int_{\R} \int_{t_n}^{t_{n+1}}\!\!\!\!\!
    \mathcal{E}^{3,n}_{\Dx} \,dtdx\Bigr|
    =\Bigl|\sum_{n}\int_{\R}\int_{t_n}^{t_{n+1}}\!\!\! (u^{n+1})_x 
    \left( \varphi v - \varphi v^{\Dx}\right)_{xx}\,dtdx\Bigr|
    \\
    &\le \norm{u^\Dx}_{L^2(0,T;H^1([-R,R]))} \Bigl( \int_0^T 
    \norm{\test v(t) - \test
      \mathcal{P}v(t)}_{H^2([-R+1,R-1])}^2\,dt\Bigr)^{1/2}
    \\ 
    &\le \Dx \,C(\norm{u_0}_{L^2(\R)},R) \norm{v}_{L^2((0,T);H^3([-R+1,R-1]))}
    \to 0, \ \text{as} \  \Dx \downarrow 0,
  \end{align*}
  and integration by parts
  \begin{align*}
   \Bigl| & \sum_{n} \int_{\R} \int_{t_n}^{t_{n+1}} \mathcal{E}^{4,n}_{\Dx}
    \,dtdx\Bigr|
    = \Bigl|\sum_{n} \int_{\R} \int_{t_n}^{t_{n+1}}  \left(u^{\Dx}-
      u^{n+1} \right)_x \left( \varphi v \right)_{xx} \,dtdx \Bigr|
    \\ 
    &\quad = 
    \Bigl|\sum_{n} \int_{\R} \int_{t_n}^{t_{n+1}}  \left( -\Dt
      D_t^{+}u^n + (t-t_n) D_t^{+}u^n \right) \left( \varphi v
    \right)_{xxx} \,dtdx\Bigr|
    \\
    &\quad \le \Dt \sum_n \int_{-R+1}^{R-1}\int_{t_n}^{t_{n+1}}\!\!\!\! \abs{D_t^+
      u^n}\,\abs{(\test v)_{xxx}}\,dtdx
    \\
    & \quad \le \Dt \,C(\norm{u_0}_{L^2(\R)},R) \norm{\test v}_{L^2((0,T);H^5([-R+1,R-1]))}
    \to 0, \ \text{as} \  \Dt \downarrow 0.
  \end{align*}
  Next, we estimate the term containing $\mathcal{E}^{5,n}_\Dx$,
  \begin{align*}
    \sum_{n}& \int_{\R} \int_{t_n}^{t_{n+1}} \mathcal{E}^{5,n}_{\Dx}
    \,dx\,dt = \sum_{n} \int_{\R} \int_{t_n}^{t_{n+1}} \left( -
      \frac{(u^{\Dx})^2}{2} +  \frac{(u^{n+1})^2}{2}  \right)(\varphi
    v)_x \,dtdx \\ 
    &=\sum_{n} \int_{\R} \int_{t_n}^{t_{n+1}} \left( -
      \frac{(u^{n})^2}{2}  +  \frac{(u^{n+1})^2}{2}   \right)(\varphi
    v)_x \,dtdx \\ 
    &\quad - \sum_{n} \int_{\R} \int_{t_n}^{t_{n+1}} \left(   u^n
      (t-t^n) D^{+}_t u^n + \frac12 (t-t^n)^2 (D_t^{+}u^n)^2
    \right)(\varphi v)_x \,dtdx \\ 
    &= \Dt \sum_{n} \int_{\R} \int_{t_n}^{t_{n+1}} \left(
      \frac12(u^{n+1} + u^n) D_t^{+} u^n \right)(\varphi v)_x \,dtdx
    \\ 
    &\quad - \sum_{n} \int_{\R} \int_{t_n}^{t_{n+1}} \left(
       u^n (t-t_n) D^{+}_t u^n + \frac12 (t-t_n)^2
      (D_t^{+}u^n)^2 \right)(\varphi v)_x \,dtdx. 
  \end{align*}
  We claim that all the terms in the above expression converges to
  zero as $\Dt$ converges to zero since;
  \begin{align*}
    \Bigl| &\sum_{n}\int_{t_n}^{t_{n+1}} \int_{-R+1}^{R-1} u^n D_t^{+} u^n
    (\varphi v)_x \,dxdt\Bigr|
    \\
    &\le \sum_{n}\int_{t_n}^{t_{n+1}} \norm{u^n}_{L^{\infty}([-R+1,
      R-1])} \norm{D_t^{+} u^n \varphi}_{H^{-2}([-R,R])}
    \norm{\varphi v}_{H^3([-R+1,R-1])}\,dt\\
    &\le C
    \sum_{n}\int_{t_n}^{t_{n+1}}\norm{u^n}_{H^1([-R,R])} \left(
      \norm{u^{n+1}_x}_{L^2([-R,R])}+1\right) \norm{\varphi
      v}_{H^3([-R+1,R-1])} \,dt\\ 
    & \le C(\norm{u_0}_{L^2(\R)}, R)
    \norm{v}_{L^{\infty}((0,T);H^3([-R+1,R-1]))}
  \end{align*}
  and similarly,
  \pagebreak
  \begin{align*}
  & \abs{\sum_{n}\int_{t_n}^{t_{n+1}}  \int_{-R+1}^{R-1}  u^{n+1} D_t^{+} u^n
    (\varphi v)_x \,dxdt} \\
& \qquad \qquad \qquad \qquad \qquad \qquad  \le C(\norm{u_0}_{L^2(\R)}, R)
    \norm{v}_{L^{\infty}((0,T);H^3([-R+1,R-1]))}.
  \end{align*}
  Furthermore
  \begin{align*}
    \Bigl|\int_\R \int_{t_n}^{t_{n+1}}
    u^n\left(t-t_n\right) & D_t^+u^n \left(\test v\right)_x
    \,dtdx\Bigr|\\
    &\le \norm{u^n}_{L^\infty([-R+1,R-1])} 
    \Dt \int_{-R+1}^{R-1}\int_{t_n}^{t_{n+1}}\abs{D_t^+
      u^n}\,\abs{(\test v)_x}\,dtdx, 
  \end{align*}
  and
  \begin{align*}
   & \Bigl| \int_\R \int_{t_n}^{t_{n+1}} \left(t-t_n\right)^2
    \left(D_t^+ u^n\right)^2 \left(\test v\right)_{x}\,dtdx\Bigr| \\
    &\qquad \le \left(\norm{u^{n+1} - u^n}_{L^\infty([-R+1,R-1])}    \right) \Dt 
    \int_{-R+1}^{R-1}\int_{t_n}^{t_{n+1}}\abs{D_t^+
      u^n}\,\abs{(\test v)_x}\,dtdx.
  \end{align*}
  Therefore, these two terms can be estimated in the same manner as
  the preceding two term.
  Hence
  \begin{align*}
    \sum_{n} \int_{\R} \int_{t_n}^{t_{n+1}} \mathcal{E}^{5,n}_{\Dx}
    \,dtdx \to 0, \ \text{as} \  \Dt \downarrow 0.
  \end{align*}
  Combining all these above estimates, we conclude that \eqref{eq:clm}
  holds. Furthermore, passing limit as $\Dx \to 0$, we
  conclude that
  \begin{equation}
    \label{eq:clm1}
    \int_0^T \int_{\R} u_t \varphi v - \frac{u^2}{2} (\varphi v)_x +
    u_x (\varphi v)_{xx} \,dx\,dt=0,
  \end{equation}
  for any test function $v \in C_c^{\infty} ([-R+1,R-1] \times
  [0,T))$.  Finally, we choose $v = \phi/\varphi$ in \eqref{eq:clm1}
  with $\phi \in C_c^{\infty} ([-R+1,R-1] \times [0,T))$ and
  integrate-by-parts to conclude that \eqref{weak solution} holds, i.e.~that
  \begin{align*}
    \int_0^T \int_{-\infty}^{\infty} \Bigl(\phi_tu + \phi_x
    \frac{u^2}{2} - \phi_{xx}u_x \Bigr)\,dxdt +
    \int_{-\infty}^{\infty} \phi(x,0)u_0(x)\,dx =0.
  \end{align*}
  This finishes the proof of the Theorem ~\ref{theo:main}.

\end{proof}


\section{Numerical experiments}
\label{sec:numerical}
The fully-discrete scheme given by \eqref{eq:scheme} has been tested
on several numerical experiments  in order to test how well this
method works in practice. 

We let $S_\Dx$ consist of piecewise cubic splines defined as follows:
Let $f$ and $g$ be the functions
\begin{align*}
  f(y)&=1+y^2\left(2\abs{y}-3\right),\\
  g(y)&=
  \begin{cases}
    y(y+1)^2  &y\le 0,\\
    -y(y-1)^2 & y>0,
  \end{cases}
\end{align*}
and we define $f(y)=g(y)=0$ for $\abs{y}>1$. For $j\in \Z$ we define
\begin{equation*}
  v_{2j}(x)=f\left(\frac{x-x_j}{\Dx}\right),\qquad
  v_{2j+1}(x)=g\left(\frac{x-x_j}{\Dx}\right),
\end{equation*}
where $x_j=j\Dx$. The space spanned by $\seq{v_j}_{j=-M}^M$ is a
$4M+2$ dimensional subspace of $H^2(\R)$. In our numerical examples,
we used periodic boundary conditions. In the examples computing
solitary waves, the exact solution, as well as the numerical
approximations are all very close to zero at the boundary. Regarding
the weight function, we chose this to be $\test(x)=50+x$ in the
intervals under consideration in all our examples. In the Newton
iteration to obtain $u^{n+1}$, \eqref{eq:iteration scheme}, we
terminated the iteration if $\norm{w^{\ell+1}-w^{\ell}}\le \Dx^2$.

For $t=n\Dt$, we set $u_{\Dx}(x,t) = u^n(x,t)=\sum_{j=-M}^M u^n_j
v_j(x)$.  In all our examples, we measured the percentage $L^2$ error,
defined as
\begin{equation*}
  E=100\,\frac{\norm{u-u_\Dx}_{L^2}}{\norm{u}_{L^2}}.
\end{equation*}

\subsection{One-soliton solution} 
The equation \eqref{eq:main} has an exact solution
\begin{equation}
  \label{eq:onesol}
  w_1(x,t)=9\left(1-\tanh^2\left(\sqrt{3/2}(x-3t)\right)\right).
\end{equation}
This represents a single ``bump'' moving to the right with speed 3. We
have tested our scheme with initial data $u_0(x)=w_1(x,-1)$ in order
to check how fast this scheme converges. Recall that we are using
$w_1(x,-1)$ as initial data, so that $w_1(x,1)$ represents the
solution at $t=2$. The solution was calculated on a uniform grid with
$\Dx=20/(2M)$ in the interval $[-10,10]$.  In Table~\ref{tab:1} we
show the relative errors as well as the numerical convergence rates
for this example.
\begin{table}[h]
  \centering
  \begin{tabular}[h]{c|r r}
    $M$ &\multicolumn{1}{c}{$E$}&\multicolumn{1}{c}{rate}\\
    \hline\\ [-2.5ex]
    8  & 61.5 & \\[-1ex]
    16 & 33.6 & \raisebox{1.5ex}{0.87}\\[-1ex]
    32 & 5.8 & \raisebox{1.5ex}{2.52} \\[-1ex]
    64 & 3.2  & \raisebox{1.5ex}{0.86}\\[-1ex]
    128 & 3.1  & \raisebox{1.5ex}{0.03}\\[-1ex]
    256  & 1.9  & \raisebox{1.5ex}{0.69}\\[-1ex]
    512  & 1.1  & \raisebox{1.5ex}{0.87}\\[-1ex]
    1024  & 0.6  & \raisebox{1.5ex}{0.94}
  \end{tabular}
  \caption{Relative percentage  $L^2$ errors for the
    one-soliton solution, $w_1(x,2)$} 
  \label{tab:1}
\end{table}
From Table~\ref{tab:1} we see that the scheme converges, that the rate
is a bit erratic, but seems to converge to one.

\subsection{Two-soliton solution}
\label{subsub:twosol}
Physically, two solitons which have different shapes move with
different velocities, which is a result of the dependence between the
height of the soliton and the velocity. A higher soliton travels
faster than a lower soliton. If the two solitons travel along a
surface, the higher soliton will overtake the lower soliton, and after
the collision, both solitons will emerge unchanged.  We use the
following test problem for the two-soliton
interaction, where $u(x,0)=w_2(x,-10)$, with
\begin{equation}
  \label{eq:twosoliton}
  w_2(x,t)=6(b-a)\frac{b \csch^2\left(\sqrt{b/2}(x-2bt)\right) 
    +a\sech^2\left(\sqrt{a/2}(x-2at)\right)}
  {\left(\sqrt{a}\tanh\left(\sqrt{a/2}(x-2at)\right) - 
      \sqrt{b}\coth\left(\sqrt{b/2}(x-2bt)\right)\right)^2},
\end{equation}
for any real numbers $a$ and $b$. We have used $a=0.5$ and $b=1$. This
solution represents two waves that ``collide'' at $t=0$ and separate
for $t>0$. For large $\abs{t}$, $w_2(\cdot,t)$ is close to a sum of
two one-solitons at different locations.

Computationally, this is a much harder problem than the one-soliton
solution. We
computed the approximate solution at $t=20$. The exact solution in
this case is $w_2(x,10)$.  Figure~\ref{fig:1} we show the exact and
numerical solutions at $t=20$. 
\begin{figure}[h]
  \centering
  \includegraphics[width=0.99\linewidth]{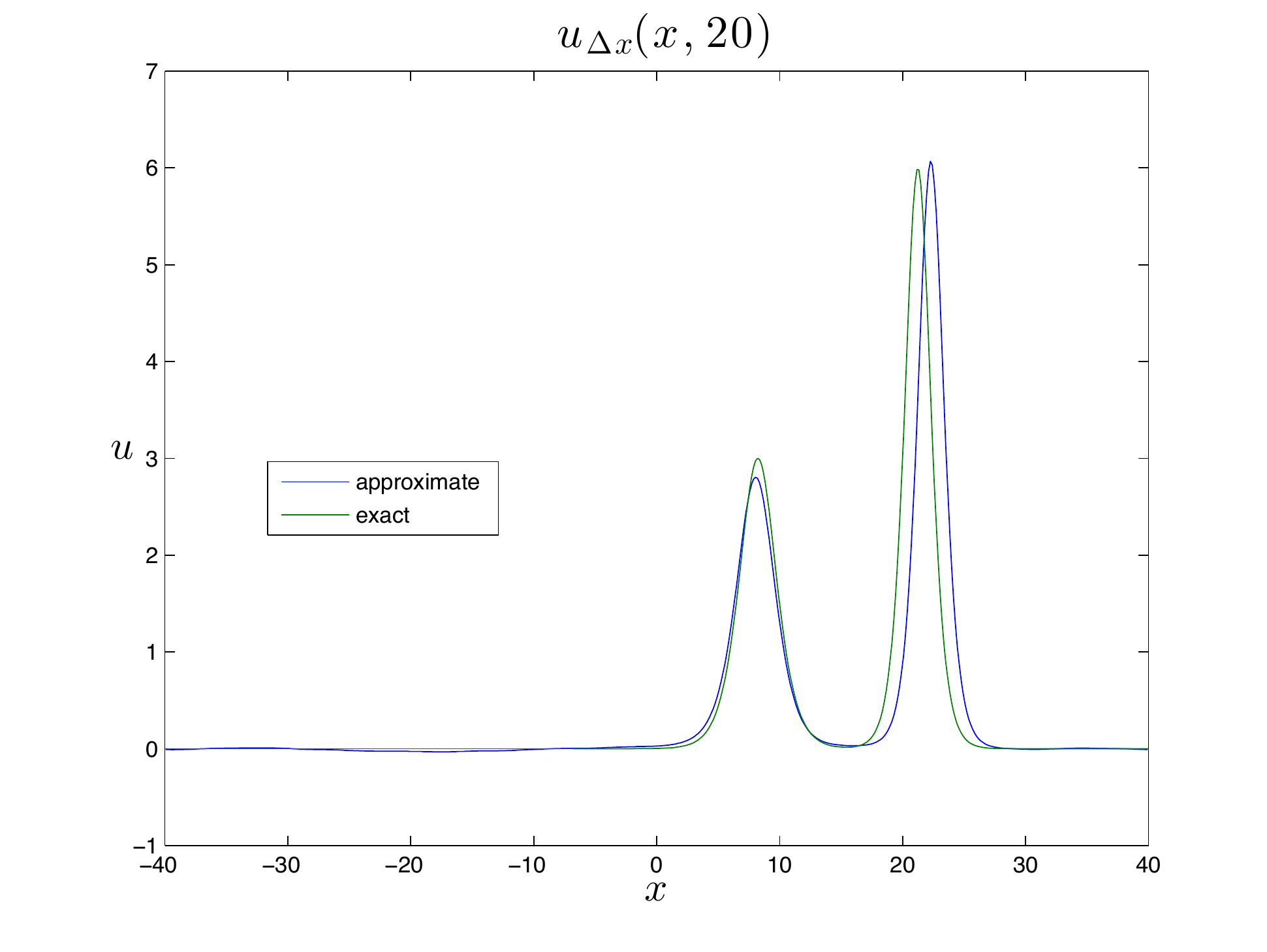}
  \caption{The exact and numerical solutions at $t=20$ with initial
    data $w_2(x,-10)$ with $M=256$.}
  \label{fig:1}
\end{figure}
The computed solution in Figure~\ref{fig:1} looks ``right'', in the
sense that the two bumps in the solution have separated well and
passed through each other. Nevertheless, the error is more than 50\%.
This is due to an error in the position of the larger bump, which
again is due to a much smaller error in the height of the bump. This
error causes the speed of the wave to be slightly larger than the
speed of the corresponding wave in the exact solution. Since the wave
is quite narrow, this causes the $L^2$ error to be large. In
Table~\ref{tab:2} we show the percentage errors for the two-soliton
simulation.
\begin{table}[h]
  \centering
  \begin{tabular}[h]{c|r r}
    $M$ &\multicolumn{1}{c}{$E$}&\multicolumn{1}{c}{rate}\\
    \hline\\ [-2.5ex]
    64  & 108 & \\[-1ex]
    128 & 41 & \raisebox{1.5ex}{1.3}\\[-1ex]
    256 & 54 & \raisebox{1.5ex}{-0.3} \\[-1ex]
    512 & 49  & \raisebox{1.5ex}{0.1}\\[-1ex]
    1024 & 30  & \raisebox{1.5ex}{0.7}\\[-1ex]
    2046 & 16 &\raisebox{1.5ex}{0.9}
  \end{tabular}
  \caption{Relative percentage  $L^2$ errors for the
    two-soliton solution.} 
  \label{tab:2}
\end{table}

\subsection{Initial data in $L^2$}\label{subsub:l2}
We have also applied our scheme on an example where the initial data
are in $L^2$, but not in any Sobolev space with positive index. To
this end we have chosen initial data
\begin{equation}
  \label{eq:l2init}
  u_0(x)=
  \begin{cases}
    0 & x\le 0,\\
    x^{-1/3} &0<x<1,\\
    0 &x\ge 1,
  \end{cases}
\end{equation}
if $x$ is in $[-5,5]$ and extended periodically outside this interval.
An exact solution is not available in this case, so we used a
third-order discontinuous Galerkin approximation with 386 degrees of
freedom as a reference solution, see \cite{ShuYan:2002,HKR1}. There is
no proof that this reference solution is close to the exact solution,
but lacking other alternatives, we choose to compare the approximate
solutions generated by our finite element scheme with this solution.

In Table~\ref{tab:L2errors} we show the relative errors for our
element method.
\begin{table}[h]
  \centering
  \begin{tabular}[h]{c|r r}
    $M$ &\multicolumn{1}{c}{$E$} &\multicolumn{1}{c}{rate} \\
    \hline\\[-2ex]
    16  & 65 & \\[-1ex]
    32 & 61 & \raisebox{1.5ex}{0.09} \\[-1ex]
    64 & 55 & \raisebox{1.5ex}{0.13} \\[-1ex]
    128 & 50  & \raisebox{1.5ex}{0.14} \\[-1ex]
    256 & 46  & \raisebox{1.5ex}{0.14} \\[-1ex]
    512 & 42  & \raisebox{1.5ex}{0.11} \\[-1ex]  
    1024 & 40  & \raisebox{1.5ex}{0.08} \\[-1ex]  
    2048 & 39  & \raisebox{1.5ex}{0.05} \\[-1ex]  
    4096 & 37  & \raisebox{1.5ex}{0.07} \\[-1ex]  
    8192 & 34  & \raisebox{1.5ex}{0.12} 
  \end{tabular}
  \caption{Relative percentage $L^2$ error between a reference
    solution using the discontinuous Galerkin method and our
    element method with initial data \eqref{eq:l2init} and $t=0.5$.} 
  \label{tab:L2errors}
\end{table}
The large errors and the slow convergence rate both indicate that we
are not yet in asymptotic regime. In Figure~\ref{fig:2} we show the
approximate solution at with the finest resolution (32768 degrees of
freedom) and the reference solution.
\begin{figure}[h]
  \centering
  \includegraphics[width=0.99\linewidth]{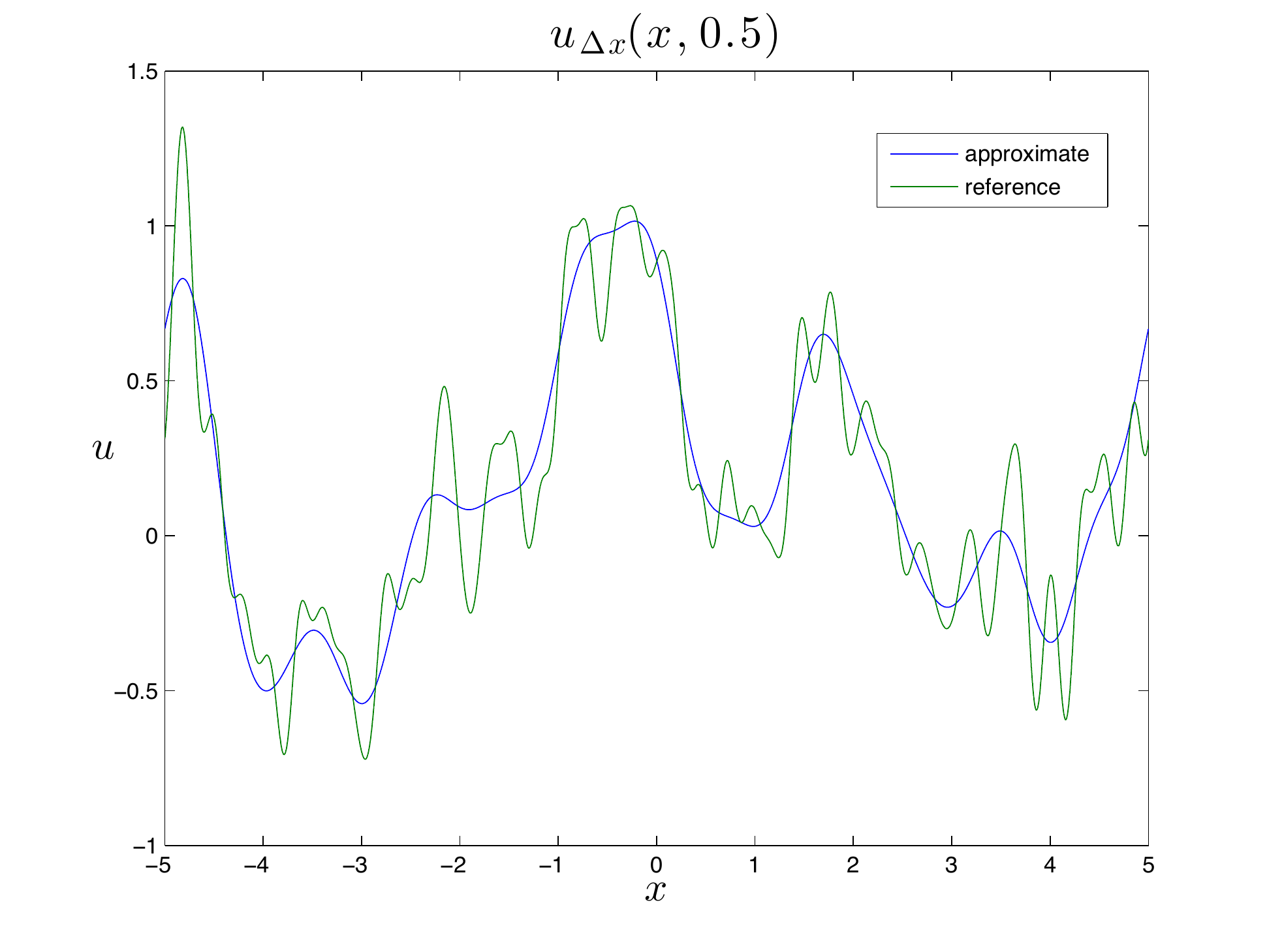}
  \caption{The numerical solution $u_{\Dx}(x,0.5)$ with initial data
    \eqref{eq:l2init} with $M=8192$, and the reference solution found
    by the third-order discontinuous Galerkin method.}
  \label{fig:2}
\end{figure}
There is however some doubt about the accuracy of the reference
solution. Our approximate solution is very close to an approximate
solution found by a simple difference scheme, see \cite{HKR1}, using
$\Dx=10/512000$.

\vspace{6.5mm}
\noindent {\bf Acknowledgments.}
This paper was written when NHR was a quest of the Seminar f\"{u}r
Angewandte Mathematik, ETH, Z\"{u}rich. This institution is thanked
for its hospitality. UK was supported in part by a Humboldt Research
Fellowship through the Alexander von Humboldt Foundation.


\end{document}